\documentclass[10pt,twoside]{siamltex}

\usepackage{amsfonts,epsfig,amsthm}
\usepackage{gauss}

\DeclareMathOperator{\Tr}{Tr}

\newtheorem{thm}{Theorem}[]
\newtheorem{cor}[thm]{Corollary}

\newtheorem{prop}[thm]{Proposition}

\newtheorem{defin}[thm]{Definition}
\newtheorem{expl}[thm]{Example}

\newtheorem*{question}{Question}

\newcommand{\reals}{\mathbb{R}}

\begin{document}



\bibliographystyle{plain}
\title{
On the maximal angle between copositive matrices
}
\author{
Felix Goldberg\thanks{Department of Mathematics, Technion-IIT, Technion City, Haifa 32000, Israel
(felix.goldberg@gmail.com).}
\and
Naomi Shaked-Monderer\thanks{The Max Stern Yezreel Valley College,  Yezreel Valley 19300, Israel (nomi@tx.technion.ac.il). This work was supported by grant no.\ G-18-304.2/2011
by the German-Israeli Foundation for Scientific Research and Development (GIF).}
}

\pagestyle{myheadings}
\markboth{F. \ Goldberg and N. \ Shaked-Monderer}{On the maximal angle between copositive matrices}
\maketitle

\begin{abstract}
Hiriart-Urruty and Seeger have posed the problem of finding the maximal possible angle $\theta_{\max}(\mathcal{C}_{n})$ between two copositive matrices of order $n$. They have
proved that $\theta_{\max}(\mathcal{C}_{2})=\frac{3}{4}\pi$ and conjectured that $\theta_{\max}(\mathcal{C}_{n})$ is equal to $\frac{3}{4}\pi$ for all $n \geq 2$. In this note
we disprove their conjecture by showing that $\lim_{n \rightarrow \infty}{\theta_{\max}(\mathcal{C}_{n})}=\pi$. Our proof uses a construction from algebraic graph theory. We
also consider the related problem of finding the maximal angle between a nonnegative matrix and a positive semidefinite matrix of the same order.
\end{abstract}


\section{Introduction}

A matrix $A$ is called \emph{copositive} if $x^{T}Ax \geq 0$ for every vector $x \geq 0$. The set of $n\times n$ copositive matrices $\mathcal{C}_{n}$  is a closed convex cone in
the space $\mathcal{S}_n$ of $n\times n$ symmetric matrices.  By the definition, the cone  $\mathcal{C}_{n}$ includes as subsets the cone $\mathcal{P}_{n}$ of positive
semidefinite matrices and the cone $\mathcal{N}_{n}$ of symmetric nonnegative matrices of order $n$. Therefore, it is easy to see that $\mathcal{P}_{n}+\mathcal{N}_{n} \subseteq
\mathcal{C}_{n}$.

In \cite{Dia62} Diananda proved that for $n \leq 4$ this set inclusion is in fact an equality, and also cited an example due to A. Horn that shows that for $n \geq 5$ there are
copositive matrices which cannot be decomposed as a sum of a positive semidefinite and a nonnegative matrix  (see also \cite[p. 597]{HirSee10}). In a remarkable recent paper
\cite{Hil12} Hildebrand has described all extreme rays of $\mathcal{C}_{5}$, but very little is known about the structure of $\mathcal{C}_{n}$ for $n \geq 6$.

Understanding the structure of this cone is important, among other reasons, since many combinatorial and nonconvex quadratic optimization problems can be  equivalently
reformulated
as linear problems over the cone $\mathcal{C}_n$ or its dual, the cone $\mathcal{C}^*_{n}$ of $n\times n$ \emph{completely positive} matrices (i.e., matrices $A$ that possess a
factorization $A=BB^T$, where $B\geq 0$).
For more information about copositive matrices and copositive optimization we refer the reader to the recent surveys \cite{HirSee10,Dur10,BomSchuch12} and the  references
therein.

This paper is dedicated to the solution of a problem posed by Hiriart-Urruty and Seeger in their survey \cite{HirSee10}:

\emph{What is the greatest possible angle between two matrices in $\mathcal{C}_{n}$?}

The angle between vectors $u,v$ in an inner product space $V$ is:
$$
\angle(u,v)=\arccos{\frac{\langle u,v \rangle}{||u||\cdot||v||}}.
$$
Given a convex cone $K\subseteq V$, the maximal angle attained between two vectors in the cone $K$ is denoted $\theta_{\max}(K)$, and a
pair of vectors attaining this angle is called \emph{antipodal}. For the study of maximal angles of cones
we refer to \cite{IusSee05,IusSee09}.

Here we consider $V=\mathcal{S}_n$, with the standard inner product
$$\langle A,B\rangle=\Tr{AB}$$
and the norm associated with it, that is the Frobenius norm $||A||=\sqrt{\sum_{i,j=1}^{n}{|a_{ij}|^{2}}}$.

In \cite{HirSee10} it was shown that $\theta_{\max}(\mathcal{C}_{2})=\frac{3}{4} \pi$ and the unique pair of $2 \times 2$ matrices (up to multiplication by a positive scalar)
that attains this angle was found. Furthermore, in \cite[Remark 6.18]{HirSee10} a somewhat hesitant conjecture was made to the effect that
$\theta_{\max}(\mathcal{C}_{n})=\frac{3}{4} \pi$ for all $n \geq 2$.

We show in this note that the authors of \cite{HirSee10} were rightly apprehensive about the said conjecture, and that the correct asymptotic answer to their problem is:
$$
\lim_{n \rightarrow \infty}{\theta_{\max}(\mathcal{C}_{n})}=\pi.
$$
Note that the cone $\mathcal{C}_{n}$ is \emph{pointed}, i.e., $\mathcal{C}_{n}\cap (-\mathcal{C}_{n})=\{0\}$ \cite[Proposition 1.2]{HirSee10},
and thus clearly $\theta_{\max}(\mathcal{C}_{n})<\pi$ for every $n$.

For the proof, we consider the maximal angle between a positive semidefinite matrix and a nonnegative matrix of the same order $n$. Let us denote this maximal angle by
$\gamma_n$, i.e.,
\[\gamma_n= \max_{0\ne X\in \mathcal{P}_{n} \atop 0\ne Y\in\mathcal{N}_{n}}  \angle(X,Y) =  \max_{ X\in \mathcal{P}_{n}, Y\in\mathcal{N}_{n}\atop ||X||=||Y||=1}\arccos\langle
X,Y\rangle .\]
This maximum exists, since both $\mathcal{N}_{n}$ and $\mathcal{P}_n$ are closed and their intersection with the unit sphere is compact. Then by the inclusion
$\mathcal{P}_{n}+\mathcal{N}_{n} \subseteq \mathcal{C}_{n}$ we have
$$ \gamma_n\le \theta_{\max}(\mathcal{P}_{n}+\mathcal{N}_{n})\le \theta_{\max}(\mathcal{C}_{n}).$$

We prove our result on $\theta_{\max}(\mathcal{C}_{n})$ by establishing
\begin{thm}\label{thm:main}
$$\lim_{n \rightarrow \infty}\gamma_n=
\lim_{n \rightarrow \infty}{\theta_{\max}(\mathcal{C}_{n})}=\pi.$$
\end{thm}

This is achieved by constructing a sequence of pairs $(P_{k},N_{k})$,  $P_k\in \mathcal{P}_{n_k}$  and $N_k\in \mathcal{N}_{n_k}$, where the orders $n_k$ tend  to infinity and
such that $\angle(P_{k},N_{k}) \rightarrow \pi$. Note that $\{\gamma_n\}$ is a non-decreasing sequence, since the angle between $N\in \mathcal{N}_n$ and $P\in \mathcal{P}_n$ is
equal to the angle between $N\oplus 0\in \mathcal{N}_{n+1}$ and $P\oplus 0\in \mathcal{P}_{n+1}$.

As the problem of calculating or estimating $\gamma_n$ is interesting in its own right, we start in Section \ref{sec:gamma_n} with some initial results on this problem, finding
$\gamma_{3}$ and $\gamma_{4}$. Though the geometry of the cones $\mathcal{P}_n$ and $\mathcal{N}_n$ is much better understood that that of $\mathcal{C}_n$, calculating
$\gamma_{n}$ is a very difficult task for $n \geq 5$. We will offer an explanation for this phenomenon by showing that the determination of $\gamma_{n}$ is closely related to the
symmetric nonnegative inverse eigenvalue problem (SNIEP). Details on SNIEP and related problems can be found in \cite{Avi} and the references of \cite{Spector}.

The main result is stated and proved in Section \ref{sec:proof} by a construction based on algebraic graph theory. The interceding Sections
\ref{sec:srg}-\ref{sec:GQ} are devoted to the introduction of the relevant tools from this theory, in order to keep this note self-contained, albeit tersely so. We conclude in
Section \ref{sec:final} with some remarks.

\section{The maximal angle between a positive semidefinite matrix and a nonnegative matrix}\label{sec:gamma_n}
In this section we consider the problem of determining maximal angle between a positive semidefinite matrix and a nonnegative matrix of the same order for its own sake. However,
the observations made in this section will also be instrumental in establishing the main result.

Every $n\times n$ symmetric matrix $A$ has a unique decomposition as a difference of two positive semidefinite matrices that are orthogonal to each other:
$$A=Q-P,  \mbox{ with } Q,P \in \mathcal{P}_{n} \mbox{ and } QP=0.$$
In fact, $Q$ is the projection of $A$ on $\mathcal{P}_{n}$ and $P$ is the projection of $-A$ on the same cone.

More explicitly, let $\Lambda $ be the multiset of eigenvalues of $A$, and for every $\lambda \in \Lambda$  denote by $E_\lambda$ the orthogonal projection on the eigenspace of
$\lambda$. Then $$A=\sum_{\lambda\in \Lambda} \lambda E_\lambda$$ is the spectral decomposition of $A$.

Denote by
$\Lambda_{+}$ and $\Lambda_{-}$ the multisets of positive and negative eigenvalues of $A$, respectively. Then
$Q=\sum_{\lambda\in \Lambda_{+}} \lambda E_\lambda$ and $P=-\sum_{\lambda\in \Lambda_{-}} \lambda E_\lambda$. In particular,
the spectrum of $Q$ consists of the elements of $\Lambda_{+}$ together with $n-|\Lambda_{+}|$ zeros and the spectrum of $P$ consists of the absolute values of the elements in
$\Lambda_{-}$ together with $n-|\Lambda_{-}|$ zeros. We refer to $Q$ and $P$ as the  \emph{positive definite part} and the \emph{negative definite part} of $A$, respectively.

If $A$ is not positive semidefinite, then obviously $A\ne 0$ and $P\ne 0$, and the cosine of the angle between $A$ and $P$ is
\begin{equation}\label{eq:aa}
\frac{\langle A,P\rangle}{||A||\cdot||P||}=
\frac{-\langle  P  ,P  \rangle}{||A||\cdot||P ||}=-\frac{ ||P  ||}{||A||}=-\frac{\sqrt{\sum_{\lambda \in \Lambda_{\bf -}}{\lambda^{2}}}}{\sqrt{\sum_{\lambda \in \Lambda
}{\lambda^{2}}}}.
\end{equation}

For every nonzero symmetric $n\times n$ matrix $A$, let us denote by $\angle(A,{\mathcal{P}_n})$  the
maximal angle between $A$ and a matrix in $\mathcal{P}_n$.
The following holds:

\medskip
\begin{prop}\label{prop:A,P}
For every $A\in \mathcal{S}_n\setminus \mathcal{P}_n$, let $P\in \mathcal{P}_n$ be the negative definite part of $A$. Then
\begin{equation}\label{eq:AP_n}\angle(A,{\mathcal{P}_n})=\angle(A,P)=\arccos\left(-\frac{\sqrt{\sum_{\lambda \in \Lambda_{\bf -}}{\lambda^{2}}}}{\sqrt{\sum_{\lambda \in \Lambda
}{\lambda^{2}}}}\right),\end{equation}
where $\Lambda$  and $\Lambda_{\bf -}$ are as described above.
Moreover, $P$ is the unique matrix in $\mathcal{P}_n$, up to multiplication by a positive scalar, which forms this maximal angle with $A$.
\end{prop}

\begin{proof}
For every $0\ne X\in \mathcal{P}_n$ we have
\begin{equation}\label{furthestpsd}\frac{\langle A,X \rangle}{||A||\cdot||X||}\geq -\frac{\langle P,X\rangle}{||A||\cdot||X||}\geq -\frac{||P||}{||A|| }=\frac{\langle
A,P\rangle}{||A||\cdot||P||},  \end{equation}
where the first inequality follows from the fact that $Q$, the positive definite part of $A$, satisfies $\langle Q,X \rangle \geq 0$,  and the second inequality from the
Cauchy-Schwarz inequality. This shows that $\angle(A,X)\le \angle(A,P)$ for every $X\in \mathcal{P}_n$. By the condition for equality in the Cauchy-Schwarz inequality, we get
that  $\angle(A,X)=\angle(A,P)$ if and only if $X$ is a positive scalar multiple of $P$.
\end{proof}

Similarly, every $A\in \mathcal{S}_n$ has a unique decomposition as a difference of two nonnegative matrices that are 
orthogonal to each other:
$$A=M-N,  \mbox{ with } M,N \in \mathcal{N}_{n} \mbox{ and }  M\circ N=0,$$
where $\circ$ denotes the entrywise product of matrices (
also often called the Hadamard product).

In fact, $M=\max(A, 0)$, with the maximum defined entrywise, is the projection of $A$ on $\mathcal{N}_n$, and $N=\max(-A,0)$ is the projection of $-A$ on that cone. We refer to
$M$ and $N$ as the \emph{positive part} and the \emph{negative part} of $A$, respectively. If $A\notin \mathcal{N}_n$, then $A,N\ne 0$, and the cosine of the angle between $A$
and $N$ is
\begin{equation}\label{eq:bb}
\frac{\langle A,N\rangle}{||A||\cdot||N||}=
\frac{-\langle  N  ,N  \rangle}{||A||\cdot||N ||}=-\frac{ ||N  ||}{||A||}=-\frac{\sqrt{\sum_{a_{ij}<0 }{a_{ij}^2}}}{\sqrt{\sum {a_{ij}^2} }}.
\end{equation}
We denote by $\angle(A,\mathcal{N}_n)$ the maximal angle between $A$ and a matrix in $\mathcal{N}_n$. Then the following holds:

\medskip

\begin{prop}\label{prop:A,N}
For every $A\in \mathcal{S}_n\setminus \mathcal{N}_n$, let $N\in \mathcal{P}_n$ be the negative  part of $A$. Then
\begin{equation}\label{eq:AN_n}\angle(A,{\mathcal{N}_n})=\angle(A,N)=\arccos\left(-\frac{\sqrt{\sum_{a_{ij}<0 }{a_{ij}^2}}}{\sqrt{\sum {a_{ij}^2} }}\right).\end{equation}
Moreover, $N$ is the unique matrix in $\mathcal{N}_n$, up to multiplication by a positive scalar, which forms this maximal angle with $A$.
\end{prop}

The proof is completely parallel to the proof of Proposition \ref{eq:AP_n}, and is therefore omitted. The next proposition
 demonstrates the computation of $\angle(P, \mathcal{N}_n)$ in a special case.

\medskip

\begin{prop}\label{prop:rank1}
Let $P\in \mathcal{P}_n\setminus \mathcal{N}_n$ have rank $1$. Then $\angle(P, \mathcal{N}_n)\le\frac{3}{4}\pi$.
Furthermore, there exists a rank 1 positive semidefinite matrix $P\in \mathcal{P}_n\setminus \mathcal{N}_n$ such that $\angle(P, \mathcal{N}_n)=\frac{3}{4}\pi$.
\end{prop}

\begin{proof}
By the assumptions, $P=uu^T$, where $u$ has both positive and negative entries. By a suitable permutation of rows and columns of $P$ we may assume that
\[u=\left[\begin{array}{r}
            v \\
            -w
          \end{array}
\right],\qquad v,w\ge 0, ~~v,w\ne 0.\]
Then
\[P=\left[\begin{array}{rr}
            vv^T & -vw^T \\
            -wv^T & ww^T
          \end{array}
\right],\]
and the negative part of $P$ is
\[N=\left[\begin{array}{cc}
            0 &  vw^T \\
             wv^T & 0
          \end{array}
\right].\]
For any two vectors $x$ and $y$,
\[||xy^T|| =\sqrt{\Tr (xy^Tyx^T)}= ||x||  ||y||.\]
Thus
\[||P|| = ||u||^2=||v||^2+||w||^2~~,~~||N||=\sqrt{2}||v||\cdot||w||,\]
and
\[\langle P,N\rangle=-2||vw^T||^2=-2||v||^2  ||w||^2.\]
Thus
\[\frac{\langle P,N\rangle}{||P||\cdot||N||}=-\frac{\sqrt{2}||v||\cdot||w||}{||v||^2+||w||^2}\ge -\frac{\sqrt{2}}{2}.\]
Equality holds in the last inequality if and only if $||v||=||w||$.
Thus $\angle(P,N)\le \frac{3}{4}\pi$, with equality if and only if $||v||=||w||$.
\end{proof}

In particular, the last proposition implies the following known result (known by the proof of Proposition 6.15 in \cite{HirSee10}, and the monotonicity of $\{\gamma_n\}$).

\medskip

\begin{cor}\label{cor:.75pi}
For every $n\ge 2$,  $\gamma_n\ge \frac{3}{4}\pi$.
\end{cor}

\medskip

We can now prove

\medskip

\begin{prop}\label{prop:NP}
Let $n\ge 2$, and let $P\in \mathcal{P}_n$ and $N\in \mathcal{N}_n$  be any two matrices such that $\angle(P,N)=\gamma_n$.
Then $\langle P, N\rangle<0$,  $\diag N=0$,
and $1\le \rank P\le n-1$.
\end{prop}
\begin{proof}
By Corollary \ref{cor:.75pi}, $\gamma_n\ge \frac{3}{4}\pi$, and thus $\langle P, N\rangle<0$. This implies that $P\notin \mathcal{N}_n$ and $N\notin \mathcal{P}_n$.
Since $\angle(P,N)$ is the maximal possible angle between a positive semidefinite and a nonnegative matrix of the same order, $N$ has
to be the nonnegative matrix forming the maximal possible angle with $P$, and $P$ has to be the nonnegative matrix forming the maximal possible angle
with $N$.

By the uniqueness parts in Propositions \ref{prop:A,P} and \ref{prop:A,N}, $N$ is a positive scalar multiple of the negative part of $P$, and $P$ is a positive
scalar multiple of the negative definite part of $N$. Since $\diag P\ge 0$ and $N$ is the negative part of $P$, we get that $\diag N=0$. By the Perron-Frobenius Theorem 
the nonzero $N$ has at least one positive eigenvalue, so its negative definite part $P$ satisfies $\rank P\le n-1$.
\end{proof}


\medskip

\begin{prop}\label{prop:rankn-1}
Let $n\ge 2$, let $N\in \mathcal{N}_n$  have $\diag N=0$ and let $P$ be its negative definite part. If $\rank P=n-1$, then
$\angle(N, \mathcal{P}_n)<\frac{3}{4}\pi$.
\end{prop}

\begin{proof}
By the assumptions on $N$,  its eigenvalues are $\rho=\lambda_1>0$, and $n-1$ negative eigenvalues $\lambda_2, \ldots , \lambda_n$ with  $\sum_{i=2}^n\lambda_i=-\rho$.  By
Proposition \ref{prop:A,P},
\[\cos \angle(N, \mathcal{P}_n)= -\frac{\sqrt{\sum_{i=2}^n{\lambda_i^{2}}}}{\sqrt{\rho^2+\sum_{i=2}^n{\lambda_i^{2}}}}.\]
The function $g(x_2, \ldots, x_n)=\sum_{i=2}^n x_i^2$  is convex, and thus attains its maximum on the compact convex set
\[\Delta=\left\{(x_2, \ldots, x_n)\in \reals ^{n-1}\, :\, x_i\le 0,~ i=2, \ldots, n-1 , \mbox{ and } \sum_{i=2}^n x_i=-\rho\right\}\]
at an extreme point of this set, i.e., at a point $x$ such that $x_i=-\rho$ for some $i$ and  $x_j=0$ for $j\ne i$. That is,
\[\max_{x\in \Delta}g(x)=\rho^2.\]
The function $f(t)=-\sqrt{\frac{t}{\rho^2+t}}$ is  decreasing  on $[0, \infty)$, and thus $f(g(x_2, \ldots, x_2))$ attains a minimum
on $\Delta$ where $g$ attains its maximum, and $\min_{x\in \Delta}f(g(x))=-\sqrt{\frac{\rho^2}{2\rho^2}}=-\frac{\sqrt{2}}{2}$. Since $\cos \angle(N, \mathcal{P}_n)=f(g(\lambda_2,
\ldots, \lambda_n))$, and
$(\lambda_2, \ldots, \lambda_n)\in \Delta$, we get that $\angle(N, \mathcal{P}_n)\le \cos (\min_{x\in \Delta}f(g(x))) =\frac{3}{4}\pi$.

By the assumption that $\rank P=n-1$  we see that $(\lambda_2, \ldots, \lambda_n)$ is not an extreme point of $\Delta$, and since
$g(x)$ is strictly convex on $\Delta$, it does not attain its maximum on $(\lambda_2, \ldots, \lambda_n)$, and neither does $\arccos(f(g(x)))$. Hence the strict inequality.
\end{proof}

In other words, Proposition \ref{prop:rankn-1} tells us that if $(N,P)$ is a pair attaining $\gamma_{n}$, then we must have $\rank P \leq n-2$.

We can now show:
\medskip

\begin{thm}\label{thm:n<5}
For $n\le 4$,  $\gamma_n=\frac{3}{4}\pi$.
\end{thm}

\begin{proof}
Propositions \ref{prop:rank1}, \ref{prop:NP} and \ref{prop:rankn-1} imply that  $\gamma_n=\frac{3}{4}\pi$ for $n\le 3$.
It remains to consider the case of $n=4$. Also, by these propositions it suffices to consider $\angle(N, \mathcal{P}_n)$ for
$N\in \mathcal{N}_4$ with $\diag N=0$ and a negative definite part $P$ of rank 2. Such $N$ has a Perron eigenvalue $\rho>0$, and
its complete set of eigenvalues is
\[\rho\ge \mu\ge 0>\lambda_3\ge \lambda_4,\]
where $\lambda_3+\lambda_4=-\rho-\mu$  and $\lambda_4\ge -\rho$. Then
\[\cos \angle(N, \mathcal{P}_n)= -\frac{\sqrt{\lambda_3^2+\lambda_4^{2}}}{ \sqrt{\rho^2+\mu^2+\lambda_3^2+\lambda_4^{2}}} .\]
Similarly to the previous proof, we note that $g(x,y)=x^2+y^2$ is a convex function, and the set
\[\Delta=\left\{(x,y)\in \reals ^2\, :\, 0\ge x\ge y\ge -\rho \mbox{ and } x+y=-\rho-\mu\right\}\]
is a compact convex set. By the assumptions on $\rho$ and $\mu$, $\Delta$ is the line segment
\[y=-\rho-\mu-x~~,~~ -\frac{\rho+\mu}{2}\le x\le -\mu\, .\]
Its extreme points are
\[(-\mu, -\rho) \mbox{ and } \left(-\frac{\rho+\mu}{2}, -\frac{\rho+\mu}{2}\right),\]
and the maximal of $g$ on $\Delta$ is the greater of
\[g(-\mu, -\rho)=\mu^2+\rho^2 \mbox{ and } g\left(-\frac{\rho+\mu}{2}, -\frac{\rho+\mu}{2}\right)=\frac{(\rho+\mu)^2}{2}.\]
Thus
\[\max_{(x,y)\in \Delta} g(x,y)=\mu^2+\rho^2,\]
and it is attained when $x=-\mu$ and $y=-\rho$.
The function $f(t)=-\sqrt{\frac{t}{\rho^2+\mu^2+t}}$ is a decreasing  function on $[0, \infty)$, and therefore $f(g(x,y))$ attains a minimum
on $\Delta$ at $(-\mu, -\rho)$, and $\min_{(x,y)\in \Delta}f(g(x))=-\sqrt{\frac{\rho^2+\mu^2}{2(\rho^2+\mu^2)}}=-\frac{\sqrt{2}}{2}$. Since
$(\lambda_3,  \lambda_4)\in \Delta$, we get that $\angle(N, \mathcal{P}_4)\le \arccos (\min_{(x,y)\in \Delta}f(g(x))) =\frac{3}{4}\pi$.
Together with Corollary \ref{cor:.75pi} this completes the proof.
\end{proof}

Note that he matrix
\[N=\left[\begin{array}{cccc}
            0 & 1 & 0 & 0 \\
            1 & 0 & 0 & 0 \\
            0 & 0 & 0 & 1 \\
            0 & 0 & 1 & 0
          \end{array}
\right]\]
has eigenvalues $1, 1, -1, -1$, and thus by the above argument $\gamma_4=\frac{3}{4}\pi$ is attained also by a pair $(N,P)$, where $P$ is the positive semidefinite part of $N$
and $\rank P=2$.
\medskip

For $n=5$ the result of Theorem \ref{thm:n<5} no longer holds:

\medskip

\begin{expl}\label{expl:n=5}{\rm
Let
\[N=\left[\begin{array}{ccccc}
              0 &  1 & 0 & 0 &  1 \\
               1 & 0 &  1 &0 & 0  \\
              0 &  1 & 0 &  1 & 0 \\
              0 & 0 &  1 & 0 &  1 \\
               1 & 0 & 0 &  1 & 0 \\
            \end{array}
          \right]\]
          be the adjacency matrix of the 5-cycle. Its eigenvalues are well known (they are easily computed by the formula for the eigenvalues of a
          circulant matrix): the simple eigenvalue 2, the positive eigenvalue $2\cos(2\pi/5)=\frac{-1+\sqrt{5}}{2}$ of multiplicity 2, and the negative
          eigenvalue $-2\cos(\pi/5)=\frac{-1-\sqrt{5}}{2}$  of multiplicity 2. Thus the negative definite part $P$ of $N$ satisfies:
\[\cos\angle(P,N)=-\frac{\sqrt{8\cos^2(\pi/5)}}{\sqrt{4+8\cos^2(2\pi/5)+8\cos^2(\pi/5)}}= -\frac{1+1/\sqrt{5}}{2}  <-\frac{\sqrt{2}}{2},\]
implying that
\[\gamma_5\ge \arccos\left(-\frac{1+1/\sqrt{5}}{2}\right)\approx 0.7575 \pi>\frac{3}{4}\pi.\]
The negative definite part of $N$ is a scalar multiple of
\[P=\left[\begin{array}{ccccc}
              1 & -\cos(\pi/5) & \cos(2\pi/5) & \cos(2\pi/5) & -\cos(\pi/5) \\
              -\cos(\pi/5) & 1 & -\cos(\pi/5) & \cos(2\pi/5) & \cos(2\pi/5)  \\
              \cos(2\pi/5) & -\cos(\pi/5) & 1 & -\cos(\pi/5) & \cos(2\pi/5) \\
              \cos(2\pi/5) & \cos(2\pi/5) & -\cos(\pi/5) & 1 & -\cos(\pi/5) \\
              -\cos(\pi/5) & \cos(2\pi/5) & \cos(2\pi/5) & -\cos(\pi/5) & 1 \\
            \end{array}
          \right].\]
}\end{expl}

Indeed, the kind of argument that we used to prove Theorem \ref{thm:n<5} is no longer sufficient for the determination of $\gamma_{n}$ for $n \geq 5$. Here we present some
considerations which explain the new difficulties which arise in the case $n \geq 5$.

Our proofs for the case $n\le 4$ involved optimization of a convex function of the non-positive eigenvalues of a  matrix $0\ne N\in \mathcal{N}_n$ with zero
diagonal, over a convex set formed by such eigenvalue-tuples. Continuing this line of proof for $n\ge 5$ would require   some information on the
possible sets of eigenvalues of a nonnegative $n\times n$ matrix with a zero diagonal.
It is known that the eigenvalues \begin{equation}\label{eq:eval}\lambda_1\ge \lambda_2 \ge \ldots\ge \lambda_n\end{equation} of a matrix  $0\ne N\in \mathcal{N}_n$ with zero
diagonal satisfy
\begin{equation}\label{eq:sniep} \lambda_1>0,~~\lambda_n\ge -\lambda_1 ~\mbox{ and }~\sum_{i=1}^n\lambda_i=0.\end{equation} 
But for $n\ge 5$ not all sequences satisfying
(\ref{eq:eval}) and (\ref{eq:sniep}) are eigenvalues of some such $N$.
The problem of determining necessary and sufficient conditions for a set of real numbers to be the eigenvalues of some $N\in \mathcal{N}_n$ with
a zero diagonal is part of the Symmetric Inverse Eigenvalue Problem (SNIEP), which is difficult and generally open.
For $n\le 4$ the conditions (\ref{eq:eval}) and (\ref{eq:sniep}) are also sufficient, by results of \cite{Fie74} and \cite{LoewyLondon}. For $n=5$  it is shown in
\cite{Spector} that necessary and sufficient conditions for (\ref{eq:eval}) to be eigenvalues of some $N\in \mathcal{N}_n$ are (\ref{eq:sniep}) together with
\begin{equation}\label{eq:specond} \lambda_2+\lambda_5\ge 0 ~\mbox{ and }~\sum_{i=1}^5 \lambda_i^3\ge 0.
\end{equation}
For $n\ge 6$ the SNIEP is still open even for trace zero matrices.

The solution of the trace-zero SNIEP for $n=5$ demonstrates a second difficulty in applying our approach, even for $n= 5$:
The last condition in (\ref{eq:specond}) is not redundant, and the set of all non-increasing $5$-tuples that are eigenvalues of a nonnegative trace zero matrix is
not convex, complicating the relevant optimization problem. It seems that  a new approach is needed for the computation of $\gamma_n$, $n\ge 5$.

For our purpose, of proving that $\lim_{n\rightarrow \infty} \gamma_n=\infty$, we will show that a judicious choice of a nonnegative matrix $N$ will allow the pair $(N, P)$,
where $P$ is the negative definite part of the nonnegative matrix $N$, to attain ever higher angles. This will be done by taking $N$ as the adjacency matrix of a  strongly
regular graph.

\section{Strongly regular graphs}\label{sec:srg}
Recall first the definition of strongly regular graphs, due originally to Bose, and the famous formula for the eigenvalues of such a graph.

\medskip

\begin{defin}[{\rm{\cite{Bose63}}}] \label{dfn:srg} {\rm
A \emph{strongly regular graph} with parameters $(n,k,a,c)$ is a
$k$-regular graph on $n$ vertices such that any two adjacent
vertices have $a$ common neighbours and any two non-adjacent
vertices have $c$ common neighbours.}
\end{defin}

For instance, observe that the pentagon $C_{5}$ is strongly regular with parameters $(5,2,0,1)$ and that the Petersen graph is strongly regular with parameters $(10,3,0,1)$.

Obviously, not every quadruple of numbers $(a,b,c,d)$ 
is the parameter vector of a strongly regular graph. A number of necessary conditions are known and may be found
in \cite[Chapter 10]{AGT2}. We will only mention the simplest one, by way of illustration:
\begin{equation}(n-k-1)c=k(k-a-1). \label{eq:srgparameters}\end{equation}
The proof is an easy exercise in double counting.

The crucial fact for us here is that the eigenvalues of the adjacency matrix of a strongly regular graphs and their multiplicities depend only on the parameters (as there may
often be many non-isomorphic graphs sharing the same parameters):
\medskip

\begin{thm}[{{\cite[Section 10.2]{AGT2}}}] \label{thm:seidel}
    Let $G$ be a connected strongly regular graph with parameters
    $(n,k,a,c)$ and let $\Delta=(a-c)^{2}+4(k-c)$.
 The eigenvalues of the adjacency matrix $A(G)$ are:
		\begin{itemize}
		\item
		$k$, with multiplicity 1.
		\item
    $\theta=\frac{(a-c)+\sqrt{\Delta}}{2}$, with multiplicity
    $m_\theta=\frac{1}{2}\left((n-1)-\frac{2k+(n-1)(a-c)}{\sqrt{\Delta}}\right)$.
    \item
		$\tau=\frac{(a-c)-\sqrt{\Delta}}{2}$, with multiplicity
    $m_\tau=\frac{1}{2}\left((n-1)+\frac{2k+(n-1)(a-c)}{\sqrt{\Delta}}\right)$.
		\end{itemize}
		\end{thm}
Note that $m_\theta$ and $m_\tau$ have to be integers, and this is another necessary condition
 the parameters $(n,k,a,c)$ have to satisfy.
		
Let us now take $N$ to be the adjacency matrix of a strongly regular graph, and let be $P$ the negative definite part of $N$. Equation \eqref{eq:aa} takes on the following form
then:
\begin{equation}\label{eq:srg}
\frac{\left< N,P \right>}{||N||\cdot ||P||}=-\sqrt{\frac{m_{\tau}\tau^{2}}{nk}}.
\end{equation}

To complete the proof of Theorem \ref{thm:main} we would now like to exhibit a family of strongly regular graphs $\{G_{n_{k}}\}$ for which the expressions of \eqref{eq:srg} tend
to $-1$ as $n_{k} \rightarrow \infty$.


\section{Generalized quadrangles}\label{sec:GQ}

\medskip
		
\begin{defin}
A \emph{generalized quadrangle} is a finite incidence structure $(\Pi,L)$ with sets $\Pi$ of points and $L$ of lines, such that:
\begin{itemize}
\item
Two lines meet in at most one point.
\item
If $u$ is a point not on line $m$, then there are a unique point $v$ on $m$ and a unique line $\ell $ such that $u$ and $v$ are on $\ell $.
\end{itemize}
\end{defin}

For basic facts about generalized quadrangles we refer to \cite[Chapter 6]{Batten}. The advanced theory is laid out in \cite{Quadrangles}. Our definition here followed \cite[p.
129]{Spectra_BH}.
		
If the generalized quadrangle $Q$ has the further property that every line is on $s+1$ points and every point is on $t+1$ lines, then we say that $Q$ is \emph{of order $(s,t)$}
and denote it by $GQ(s,t)$. By \cite[Theorem 6.1.1]{Batten} all generalized quadrangles are either of this form or isomorphic to a grid or to a dual of a grid.

It is not known what are all the pairs $(s,t)$ for which a generalized quadrangle $G(s,t)$ exists. But
the so-called ``classical" constructions of generalized quadrangles, originally due to Tits,  yields specimens of $GQ(s,1)$, $GQ(s,s)$ and $GQ(s,s^{2})$ whenever $s=q$ is a prime
power. (cf. \cite[p. 118]{Batten} and \cite[pp. 130-131]{Spectra_BH} for descriptions of these constructions.)

We need to introduce one final concept.
The \emph{collinearity graph} $C_{Q}$ of a generalized quadrangle $Q=(\Pi,L)$ has $\Pi$ for its vertex set and $u,v \in \Pi$ are adjacent in $C_{Q}$ if and only if $u$ and $v$
lie on a line in $Q$.

\medskip

\begin{thm}[{{\cite[Theorem 9.6.2]{Spectra_BH}}}]\label{thm:graph}
Let $Q$ be a generalized quadrangle of order $(s,t)$ and let $C_{Q}$ be its collinearity graph. Then $C_{Q}$ is strongly regular with parameters
$(n,k,a,c)=((s+1)(st+1),s(t+1),s-1,t+1)$ and its spectrum is:
\begin{itemize}
\item
$k=s(t+1)$ with multiplicity $1$.
\item
$\theta=s-1$ with multiplicity $m_{\theta}=st(s+1)(t+1)/(s+t)$.
\item
$\tau=-(t+1)$ with multiplicity $m_{\tau}=s^{2}(st+1)/(s+t)$.
\end{itemize}
\end{thm}

\section{Piecing everything together}\label{sec:proof}



\begin{proof}[Proof of Theorem \ref{thm:main}]
Let $\{n_{k}\}$ be the sequence of prime powers. For each $q \in \{n_{k}\}$ there exists a classical generalized quadrangle $Q_{k}$ of the $GQ(q,q^{2})$ type. Let $N_{k}$ be the
adjacency matrix of $C_{Q_{k}}$ and let $P_{k}$ be the projection of $(-N_{k})$ on $\mathcal{P}_{n}$. Then the angle between $N_{k}$ and $P_{k}$ can be calculated with the help
of Theorem \ref{thm:graph} and \eqref{eq:srg}: its cosine is
\begin{equation}\label{eq:monty}
-\sqrt{\frac{m_{\tau}\tau^{2}}{nk}}=-\sqrt{\frac{s(t+1)}{(s+1)(s+t)}}= -\frac{\sqrt{ q^{2}+1 }}{ q+1  }
\end{equation}
and this leads to
$$
\angle(N_{k},P_{k})=\arccos\left(-\frac{\sqrt{ q^{2}+1 }}{ q+1  } \right)  \underset{q \rightarrow \infty}{\longrightarrow} \arccos{(-1)}=\pi.
$$
Since
\[\pi>\theta_{\max}(\mathcal{C}_{n_k})\ge\gamma_{n_k}\ge \angle(N_{k},P_{k})   \mbox{ for every } k, \]
this implies $\lim_{k\rightarrow\infty} \theta_{\max}(\mathcal{C}_{n_k})=\lim_{k\rightarrow\infty}\gamma_{n_k}=\pi$,
and by the monotonicity of the sequences $\{\theta_{\max}(\mathcal{C}_{n})\}$ and $\{\gamma_n\}$ the result follows.
\end{proof}

Note that we did not actually find the value of $\gamma_n$ for every $n$, which is why we refer to our result as the asymptotic solution of the Hiriart-Urruty and Seeger
problem.

To get a feel for the sequence of angles $\{\angle(N_{k},P_{k})\}$, we list here the first five
elements in the sequence. The first five prime  powers (our $q$'s) are 2, 3, 4 , 5 and 7. The first five orders of the matrix pairs we generate are:
$n_1=27$, $n_2=112$, $n_3=325$, $n_4=756$ and $n_5=2752$ ($n=(q+1)(q^3+1)$). Table 1 shows the  lower bounds on $\gamma_n$ (and thus on
$\theta_{\max}(\mathcal{C}_n)$) for these orders,  computed using (\ref{eq:monty}).

\noindent \setlength{\tabcolsep}{4pt}
\begin{tabular}{c|c|c|c|c}
    $n=27$ & $n=112$ & $n=325$ &$n=756$ & $n=2752$\\ \hline
        \!$\arccos\left(-\frac{\sqrt{5}}{3}\right)$\! &\! $\arccos\left(-\frac{\sqrt{10}}{4}\right)$ \! &
     \! $\arccos\left(-\frac{\sqrt{17}}{5}\right) $ \!&\! $\arccos\left(-\frac{\sqrt{26}}{6}\right) $\! &\! $\arccos\left(-\frac{\sqrt{50}}{8}\right)$\!\\
  $\approx 0.7677\pi$  & $\approx 0.7902\pi$ & $\approx 0.8086\pi$ & $\approx 0.8232\pi$ & $\approx 0.8451\pi$\\
\multicolumn{5}{c}{}\\
 \multicolumn{5}{c}{Table 1: Lower bounds on $\gamma_n$  and $\theta_{\max}(\mathcal{C}_n)$}
\end{tabular}

\section{A few remarks}\label{sec:final}
\begin{enumerate}
\item Theorem \ref{thm:main} implies  that for large $n$  there exist a nonnegative matrix and a positive semidefinite  matrix that are almost opposite, and
the cones $\mathcal{P}_{n}+\mathcal{N}_{n}$ and $\mathcal{C}_n$ are ``barely pointed".

\item We do not know whether the pair $(N_k, P_k)$ constructed is actually antipodal in either $\mathcal{C}_{n_k}$ or $\mathcal{P}_{n_k}+\mathcal{N}_{n_k}$. However, it is not
    hard to check that this pair satisfies the weaker property of being a \emph{critical pair} in $\mathcal{P}_{n_k}+\mathcal{N}_{n_k}$, as defined in \cite[Definition
    6.11]{HirSee10}.
     Any antipodal pair is critical but not all critical pairs are antipodal. It is not obvious that this pair is a critical pair for $\mathcal{C}_{n_k}$.

\medskip

\begin{question}
Is $\theta_{\max}(\mathcal{C}_{n})=\gamma_{n}$?
In other words, is the maximal angle in $\mathcal{C}_{n}$ always achieved by a nonnegative matrix and a positive semidefinite matrix?
\end{question}
\medskip
In fact, we do not even know the answer to the following, ostensibly simpler, question:

\begin{question}
Is $\theta_{\max}(\mathcal{P}_{n}+\mathcal{N}_n)=\gamma_{n}$?
\end{question}
\medskip
This is true for $n=2$ by the results of \cite{HirSee10}.

\item
Hiriart-Urruty and Seeger \cite[Proposition 6.15]{HirSee10} found that the (unique up to multiplication by a positive scalar) pair of antipodal matrices in $\mathcal{C}_{2}$
is:
$$
X=\frac{1}{2}\left(\begin{array}{cc}
1&-1\\
-1&1
\end{array}\right), Y=\frac{\sqrt{2}}{2}\left(\begin{array}{cc}
0&1\\
1&0
\end{array}\right).
$$
This example is in fact a special case of our construction: $Y$ can be thought of as the normalized adjacency matrix of the complete graph $K_{2}$ and $X$ is the negative
definite part of $Y$.
The right-hand side of \eqref{eq:aa} equals $-\frac{\sqrt{2}}{2}$ in this case, as can be easily verified.

We observe  that pairs of matrices that yield $-\frac{\sqrt{2}}{2}$ in \eqref{eq:aa}, and thus an angle of $\frac{3 }{4}\pi$, can be easily constructed for every order by
taking $N$ as the adjacency matrix of a bipartite graph, by the Coulson-Rushbrooke Theorem on the symmetry of their spectra (cf. \cite[p. 11]{AGT1}).

Another kind of pair which
attains the angle  $\frac{3 }{4}\pi$  can be constructed for a prime power $q$ by taking $n=(q+1)(q^2+1)$ and letting $N$ be the adjacency matrix of $C_{GQ(q,q)}$, which is
clearly not a bipartite graph.


\end{enumerate}

\section{Acknowledgements}
We would like to thank the referee for helpful comments and especially for urging us to investigate $\gamma_{n}$ beyond what has been done in the original version of the paper, thus leading us to the discovery of Theorem \ref{thm:n<5}.

\bibliography{nuim,srg}
\end{document}